\documentclass[11pt,a4paper]{article}
\usepackage{amsmath}
\usepackage{amsthm}
\usepackage{amsfonts}
\usepackage{amssymb}
\usepackage{authblk} 
\usepackage{usbib}
\usepackage{tikz}
\usetikzlibrary{positioning}
\usetikzlibrary{calc}
\newcommand*\circled[1]{\tikz[baseline=(char.base)]{
            \node[shape=circle,draw,inner sep=2pt] (char) {#1};}}
            
\usepackage{etoolbox}

\usepackage{graphicx}
\usepackage[margin=2.5cm]{geometry}
\usepackage{hyperref}
\usepackage{enumitem}

\newtheorem{theorem}{Theorem}[section]
\newtheorem{definition}[theorem]{Definition}
\newtheorem{lemma}[theorem]{Lemma}
\newtheorem{proposition}[theorem]{Proposition}
\newtheorem{corollary}[theorem]{Corollary}

\theoremstyle{remark}

\DeclareMathOperator{\Peri}{Peri}
\DeclareMathOperator{\diam}{diam}
\DeclareMathOperator{\Cent}{Cent}

\author{Andry N. Rabenantoandro\footnote{email: andry@aims.ac.za}}
\affil{Pretoria, South Africa} 
\title{Peripheral hyper-Wiener index of a graph}
\date{}

\begin{document}
\maketitle

\begin{abstract}
 In this note, we introduce a new topological index  of a graph $G$ that we term  peripheral hyper-Wiener index, denoted $PWW(G)$.  It is a natural extension of the peripheral Wiener index $PW(G)$ initiated in \cite{NKL} and  is to the peripheral Wiener index what the hyper-Wiener index is to the Wiener index. We investigate its basic properties. We compute the peripheral hyper-Wiener index of the cartesian product  and trees. In particular, we get an explicit formula for the case of the hypercubes. We also give lower and upper bounds On $PW(G)$ and $PWW(G)$ in terms of the order, size, diameter and the number of peripheral vertices. This paper is an echo to \cite{NKL}, most of the results we get are analogues of the ones therein. 
\end{abstract}

\section{Introduction}
Let $G=(V(G),E(G))$ be a non trivial, finite, simple, undirected and connected graph where $V(G)$ is its set of vertices and $E(G)$ is its set of edges. Throughout, unless otherwise stated, when we say connected graphs we mean non trivial, finite, simple, undirected and connected graph~s. An edge $\{u, v\}$ of $G$ will be denoted by $uv$ or $[u,v]$. The order of $G$ is $|V(G)|$ and its size is $|E(G)|$.  We write $K_n$ for the complete graph of order $n$, $P_n$ for the path graph of order $n$, $C_n$ for the cycle graph of order $n$, and $K_{m,n}$ for the complete bipartite graph with partite sets of order $m$ and $n$. For $u, v \in V(G)$, we denote  the distance between the vertices $u$ and $v$ of $G$ by $d(u,v|G)$ or simply $d(u,v)$ if no confusion is likely to arise. For $v\in V(G)$, the eccentricity of $v$ is $e_G(v) = e(G):= \max\{d(v,u)| u\in V(G)\setminus\{v\}\}$. The radius $r(G)$ of $G$ is the minimum eccentricity and the diameter $\diam(G)$ of $G$ is the maximum eccentricity. A vertex with minimum eccentricity is called a \textit{central vertex} and one with maximum eccentricity is called a \textit{peripheral vertex}. The center of $G$ is the set of all central vertices of $G$, denoted $\Cent(G)$. A graph $G$ is said to be \textit{self-centered} if $\Cent(G) = V(G)$. We denote by $\Peri(G)$ the set of all peripheral vertices of $G$. The periphery of $G$, $P(G)$, is the subgraph induced by the vertices in $\text{Peri}(G)$. A graph is said to be \textit{peripheral} if $G = P(G)$.

Given two graphs $G$ and $H$, the cartesian product of $G$ and $H$ is the graph $G \times H = (V(G)\times V(H), E(G \times H ))$   where $[(a,x),(b,y)]$ is an edge of $G \times H$  if either $a=b$ and $xy\in E(H)$, or $ab \in E(G)$ and $x=y$. 

The Wiener index is the first distance-based topological index of a graph. It was introduced by Harry Wiener in 1947 while studying the boiling temperatures of alkanes among other things \citep{Wiener1947, Wiener1947, WagWan18}. Subsequently, many authors studied and explored variants of the Wiener index not only for the sake of applications to chemistry but also simply for  mathematical  interests. For further details, we refer the reader to \citep{DRG01} and the references therein .The \textit{Wiener} index of a graph $G$ is defined as 
\[
W(G):= \sum_{\{u, v\}\subseteq V(G)}d(u,v) = \dfrac{1}{2}\sum_{v\in V(G)}d_G(v)
\]
where $d_G(v)$ is the sum of the distances between $v$ and all the other vertices of $G$.

In 1993, Randi\'c  \citep{Randic93} introduced a new distance-based topological index similar to the Wiener index called the hyper-Wiener index. The \textit{hyper-Wiener} index of a graph $G$ is defined as 
\[
WW(G):= \frac{1}{2}\sum_{\{u, v\}\subseteq V(G)}\left( d(u,v)+d(u,v)^2 \right).
\]

 In \cite{NKL}, Narayankar and Lokesh introduced the following variant of the Wiener index called the \textit{peripheral Wiener} index of a graph $G$
\[
PW(G):= \sum_{\{u, v\}\subseteq \Peri(G)}d(u,v) = \sum_{1\leq i<j\leq k}d(v_i, v_j)=\dfrac{1}{2}\sum_{v\in \Peri(G)}d_{P(G)}(v)
\]
where $v_1, v_2, \cdots, v_k$ are the peripheral vertices of $G$ and $d_{P(G)}(v)$ is the \textit{peripheral distance number} of $v$, defined as 
\begin{align}
d_{P(G)}(v) := \sum_{u\in \Peri(G)}d(u,v|G). \label{PeriDistNr}
\end{align}

In \cite{GFP09}, Gutman, Furtula and Petrovi\'c introduced a distance-based invariant for trees, namely the \textit{terminal Wiener} index of a tree defined as
\[
TW(T):= \sum_{1\leq i<j \leq k} d(v_i,v_j|T)
\]
where $v_1, v_2, \cdots, v_k$ are the pendant vertices of $T$. Similarly, we define the \textit{hyper terminal Wiener} index of a tree as
\[
TWW(T):=\frac{1}{2} \sum_{1\leq i<j \leq k} \left(d(v_i,v_j|T) + d(v_i,v_j|T)^2\right).
\]

Note that the two indices mentioned above may be defined for connected graphs in general.

In this note, we study  yet another variant of the Wiener index that arise naturally from the previous ones. 
\begin{definition}
The \textit{peripheral hyper-Wiener} index of $G$ is defined as
\[
PWW(G):=\frac{1}{2} \sum_{\{u, v\}\subseteq \Peri(G)}\left(d(u,v) + d(u,v)^2\right) = \dfrac{1}{4}\sum_{v\in \Peri(G)}\left(d_{P(G)}(v) + d_{P(G)}(v)^2 \right).
\]
\end{definition}

The following are easy observations:
\begin{proposition}
For $n\geq m \geq 2$, we have:
\begin{enumerate}[label = (\arabic*)]
\item $PWW(K_n) = \displaystyle{{n}\choose{2}}$.
\item $PWW(K_{1,n}) = 3 \displaystyle{{n}\choose{2}}$.
\item $PWW(K_{m,n}) = 3\displaystyle{{n}\choose{2}} + 3 \displaystyle{{m}\choose{2}} + nm.$
\item \label{EasyProp4} Let $G$ be a graph with exactly $r$ peripheral vertices. Then $
PWW(G) \geq \displaystyle{{r}\choose{2}}$, with equality if and only if $G=K_r$.
\end{enumerate}
\end{proposition}
\begin{proof}
\ref{EasyProp4} Let $v_1, v_2, \cdots, v_r$ be the peripheral vertices of $G$. We have
\begin{equation} \label{PWWlowerbound}
PWW(G) = \frac{1}{2}\sum_{1\leq i < j \leq r} \left( d(v_i, v_j) +  d(v_i, v_j)^2\right) \geq \frac{1}{2} \sum_{1\leq i < j \leq r}2 = \displaystyle{{r}\choose{2}}.
\end{equation}
Assume that $PWW(G) = \displaystyle{{r}\choose{2}}$. Necessarily, for $1\leq i<j \leq r$, we have that $  d(v_i, v_j) +  d(v_i, v_j)^2 = 2$ so that $\diam(G)=1$. Since $G$ is connected, it is the complete graph $K_r$.
\end{proof}

\section{Comparison with $W(G), PW(G)$ and $WW(G)$}
 The partially ordered set $\left(\{W(G), PW(G),WW(G), PWW(G) \},\leq \right)$ is associated to the following Hasse diagram

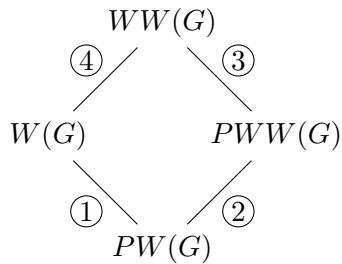
\begin{center}
\begin{tikzpicture}

    
    \node (A) at (0,-1.5) {$PW(G)$};
    \node (B) at (-1.5,0) {$W(G)$};
    \node (C) at (1.5,0) {$PWW(G)$};
    \node (D) at (0,1.5) {$WW(G)$};

    \draw (A) -- (B);
    \draw (A) -- (C);
    \draw (B) -- (D);
    \draw (C) -- (D);

    
    \node[draw, circle, inner sep=1pt] (N1) at (-1, -1) {1}; 
    \node[draw, circle, inner sep=1pt] (N2) at (1, -1) {2};  
    \node[draw, circle, inner sep=1pt] (N3) at (1, 1) {3};   
    \node[draw, circle, inner sep=1pt] (N4) at (-1, 1) {4};  

\end{tikzpicture}
\end{center}
\begin{proposition}
The graph $G$ is complete if and only if $W(G)= PW(G)= WW(G)=PWW(G)$.
\end{proposition}
\begin{proof}
The inequalities \circled{1} and \circled{3} become equalities if and only if $G$ is peripheral, and the inequalities  \circled{2} and \circled{4} become equalities if and only if $G$ is complete.
\end{proof}
 In general, $W(G)$ and $PWW(G)$ are not comparable. Indeed, $W(P_3) = 4 > PWW(P_3) = 3$. However, if $T$ is the tree illustrated in Figure \ref{NonCompInd} then $W(T) = 16 < PWW(T) = 18$. Equality also holds, for instance $W(P_2) = PWW(P_2)$.

\begin{figure}[!h]
\begin{center}
\begin{tikzpicture}
    \node[draw, fill=black, circle, inner sep=1.5pt, label=above:$v_1$] (v1) at (0,0) {};
    \node[draw, fill=black, circle, inner sep=1.5pt, label=below:$v_2$] (v2) at ($(v1) + (-1,-1)$) {};
    \node[draw, fill=black, circle, inner sep=1.5pt, label=below:$v_3$] (v3) at ($(v1) + (-0.33,-1)$) {};
    \node[draw, fill=black, circle, inner sep=1.5pt, label=below:$v_4$] (v4) at ($(v1) + (0.33,-1)$) {};
\node[draw, fill=black, circle, inner sep=1.5pt, label=below:$v_5$] (v5) at ($(v1) + (1,-1)$) {};
   
    \draw (v1) -- (v2);
    \draw (v1) -- (v3);
    \draw (v1) -- (v4);
    \draw (v1) -- (v5);

\end{tikzpicture}
\caption{A tree with peripheral hyper-Wiener index strictly larger than its Wiener index. }
\label{NonCompInd}
\end{center}
\end{figure}
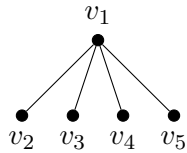

\begin{proposition}
For any connected graph $G$, $PWW(G) = WW(G) = TWW(G)$ if and only if $PW(G) = W(G) = TW(G)$ if and only if $G \simeq P_2$.
\end{proposition}

Next, we give bounds on the peripheral hyper-Wiener index in terms of the hyper-Wiener index, the order and the number of peripheral vertices.
\begin{theorem}\label{Bounds}
Let $G$ be a graph of order $n$, diameter $d$ and $|\Peri(G)|=k$. Then
\[
WW(G)-\frac{d(d-1)}{2}\left[ {\displaystyle{n}\choose{2}} - {\displaystyle{k}\choose{2}}\right] \leq PWW(G) \leq WW(G) - {\displaystyle{n}\choose{2}}+ {\displaystyle{k}\choose{2}}.
\]
\end{theorem}
\begin{proof}
By definition
\[
PWW(G) = \frac{1}{2}\sum_{\{u,v\}\subseteq \Peri(G)} \left(d(u,v) + d(u,v)^2 \right) = WW(G) - \frac{1}{2}\sum_{\{u,v\}\nsubseteq \Peri(G)} \left(d(u,v) + d(u,v)^2 \right)
\]
and for $\{u,v\}\nsubseteq \Peri(G)$, we have $1\leq d(u,v) \leq d-1$.
\end{proof}
As an immediate consequence of Theorem \ref{Bounds}, we have the following
\begin{corollary}\label{CorDiamTwo}
Let $G$ be a graph of order $n$, of diameter 2 and $|\Peri(G)|=k$. Then
\[
PWW(G) = WW(G)- {\displaystyle{n}\choose{2}}+ {\displaystyle{k}\choose{2}}.
\]
\end{corollary}
Just as for the Wiener index \cite{GutmanYehChen94}, there are no graphs of hyper-Wiener index $2$ or $5$. For any positive integer $k$, there exists a graph $G$, namely the path graph $P_{k+1}$, such that $PW(G) = k$. A quick verification would reveal that not every integer other than $2$ and $5$ arise as the peripheral hyper-Wiener index of a graph. Characterizing the set of such integers is the inverse problem for the hyper-Wiener index, and it is an open problem as far as we know.

\section{The peripheral hyper-Wiener index of a graph with diameter at most 2.}
The non trivial, finite, simple, undirected and connected graphs of size $n$ and diameter $d=1$ are the complete graphs $K_n$. We focus on the case $d=2$.
\begin{theorem}\label{DiamTwoThm}
Let $G$ be a graph of order $n$, size $m$, diameter two and $|\Peri(G)| = k$. Then
\[
PWW(G) = 2\displaystyle {{n}\choose{2}} + \displaystyle {{k}\choose{2}} -2m.
\]
\end{theorem}
\begin{proof}
We have by definition
\begin{align*}
WW(G) & = \frac{1}{2} \left(\sum_{\substack{d(u,v) = 1 \\ \{u, v\}\subseteq V(G)}} \left(d(u,v) + d(u,v)^2\right) \right) + \frac{1}{2} \left(\sum_{\substack{d(u,v) = 2 \\ \{u, v\}\subseteq V(G)}}  \left(d(u,v) + d(u,v)^2\right)\right) \\
 & =   3\displaystyle {{n}\choose{2}} -2m .
\end{align*}
Now, by Corollary \ref{CorDiamTwo}, we get the result.
\end{proof}
The proof of Theorem 2 in \citep{NKL}, which is the analogue of Theorem \ref{DiamTwoThm}, does not adapt to the case of the peripheral hyper-Wiener index. Our proof is simpler and also works for the peripheral Wiener index case as an analogue of Corollary \ref{CorDiamTwo} is also available (Corollary 1 in \citep{NKL}).

The converse of Theorem \ref{DiamTwoThm} is not true as shown by considering  the tree $T$ in Figure \ref{CounterExConvTree}.

\begin{figure}[!h]
\begin{center}
\begin{tikzpicture}[every node/.style={circle, fill=black, inner sep=2pt}]
    \node (r) at (0,0) {};          
    \node (a) at (-1.2,-1.5) {};    
    \node (b) at (0,-1.5) {};       
    \node (c) at (1.2,-1.5) {};     
    \node (d) at (1.2,-3) {};       

    \draw (r) -- (a);
    \draw (r) -- (b);
    \draw (r) -- (c);
    \draw (c) -- (d);

\end{tikzpicture}
\caption{A tree $T$ with $PWW(T) = 15$ with $\diam(T) \neq 2$.}\label{CounterExConvTree}
\end{center}
\end{figure}
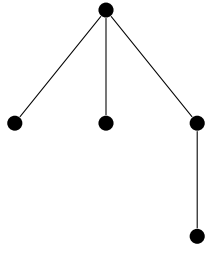

In \citep{NKL}, the authors posed the following problem: Let $G$ be a graph with diameter $\diam(G)\geq 3$ and $k$ peripheral vertices. Find upper bound for the peripheral Wiener index of $G$. 

\section{Bounds on the peripheral Wiener and peripheral hyper-Wiener index of a graph of diameter at least 3.}
In this section we provide loose bounds on the peripheral Wiener and peripheral hyper-Wiener indices of a graph, in terms of its order, size, diameter and number of peripheral vertices. Finding sharp bounds is still an open problem for general graphs. Extremal problems on peripheral Wiener index of trees are studied in \citep{Chen18}.
\begin{theorem}
Let $G$ be a graph of order $n$, of size $m$, of diameter $d$ and with $k$ peripheral vertices. Then
\begin{align*}
d\Big\lceil \frac{k}{2}\Big\rceil -(d-3)\left[{{n}\choose{2}}-{{k}\choose{2}}\right] - m \leq PW(G) \leq (d-1){{n}\choose{2}} + (d+1){{k}\choose{2}} - (d-2)m - (d-1)\Big\lceil \frac{k}{2}\Big\rceil.
\end{align*}
\end{theorem}
\begin{proof}
It suffices to notice that there are at least $\Big\lceil \dfrac{k}{2}\Big\rceil$ pairs of peripheral vertices $(u,v)$ satisfying $d(u,v) = d$ and at most $\displaystyle{{k}\choose{2}}$ such vertices. This implies that
\begin{align*}
d\Big\lceil \dfrac{k}{2}\Big\rceil \leq \sum_{\substack{d(u,v) = d \\ \{u, v\}\subseteq V(G)}}d(u,v) \leq  d\displaystyle{{k}\choose{2}},
\end{align*}
and then 
\begin{align*}
2\left[{{n}\choose{2}} - {{k}\choose{2}} -m \right] \leq \sum_{\substack{2\leq d(u,v) \leq d-1 \\ \{u, v\}\subseteq V(G)}}d(u,v) \leq (d-1)\left[{{n}\choose{2}} - \Big\lceil \dfrac{k}{2}\Big\rceil -m \right]. 
\end{align*}
Therefore
\begin{align*}
d\Big\lceil \dfrac{k}{2}\Big\rceil + 2\left[{{n}\choose{2}} - {{k}\choose{2}} -m \right] + m \leq W(G) \leq d\displaystyle{{k}\choose{2}} + (d-1)\left[{{n}\choose{2}} - \Big\lceil \dfrac{k}{2}\Big\rceil -m \right] + m.
\end{align*}
We get the desired result from Theorem 1 in \citep{NKL}:
\[
W(G)-(d-1)\left[ {\displaystyle{n}\choose{2}} - {\displaystyle{k}\choose{2}}\right] \leq PW(G) \leq W(G) - {\displaystyle{n}\choose{2}}+ {\displaystyle{k}\choose{2}}.
\]
\end{proof}
With a similar argument and using Theorem \ref{Bounds}, we get:
\begin{theorem}
Let $G$ be a graph of order $n$, of size $m$, of diameter $d$ and with $k$ peripheral vertices. Then
\begin{align*}
\frac{d(d+1)}{2}\Big\lceil \frac{k}{2}\Big\rceil +\frac{6-d(d-1)}{2}\left[{{n}\choose{2}}-{{k}\choose{2}}\right] - 2m 
\leq 
PWW(G) 
 \leq &
\frac{d(d-1)-2}{2}{{n}\choose{2}} + \frac{d(d+1)+2}{2}{{k}\choose{2}}\\
& - \frac{2-d(d-1)}{2}m  - \frac{d(d-1)}{2}\Big\lceil \frac{k}{2}\Big\rceil.
\end{align*}
\end{theorem}
\section{Peripheral hyper-Wiener index of the cartesian product}
In this section, we look at the peripheral hyper-Wiener index of cartesian products of graphs.
\begin{lemma}[\protect{\cite[Lemma 12.2]{IKR08}}]\label{DistOfProd}
    Let $G$ be the cartesian product $\prod_{i=1}^k G_i$ of connected graphs $G_i$ and let $u=(u_1, \cdots, u_k)$   and $v=(v_1, \cdots, v_k)$ be two vertices of $G$. Then
    \[
    d(u,v|G) = \sum_{i=1}^k d(u_i, v_i|G_i).
    \]
\end{lemma}
An immediate consequence of the above lemma is the following
\begin{corollary}\label{PeriOfProd}
   Let $G$ be the cartesian product $\prod_{i=1}^k G_i$ of connected graphs $G_i$. Then 
   \[
   \Peri(G) = \prod_{i=1}^k \Peri(G_i).
   \]
\end{corollary}
The peripheral Wiener index of the cartesian product of two graphs is given in terms of the peripheral Wiener indices of each factors and their respective number of peripheral vertices as follows
\begin{theorem}[\protect{\cite[Theorem 3]{NKL}}]\label{PWProd}
    Let $G_i$ be a graph  with $k_i$ peripheral vertices for $i=1, 2$. Then
    \[
    PW(G_1 \times G_2) = k_2^2 \ PW(G_1)   +  k_1^2 \ PW(G_2).
    \]
\end{theorem}
The next result gives the peripheral hyper-Wiener index of a product in terms of the peripheral hyper-Wiener indices, peripheral Wiener indices of each factors and their respective number of peripheral vertices. 
\begin{theorem}\label{PWWProd}
    Let $G_i$ be a graph  with $k_i$ peripheral vertices for $i=1, 2$. Then
    \[
    PWW(G_1 \times G_2) =  k_2^2\ PWW(G_1)   +   k_1^2\ PWW(G_2)  + 2 \ PW(G_1)   \ PW(G_2).
    \]
\end{theorem}
\begin{proof}
    Let $u_1, \cdots u_{k_1}$ be the peripheral vertices of $G_1$ and $v_1, \cdots, v_{k_2}$ that of $G_2$. The peripheral hyper-Wiener indices of $G_1$ and $G_2$ are given by
    \[
    PWW(G_1) = \frac{1}{4}\sum_{i=1}^{k_1}\sum_{j=1}^{k_1} \left(d(u_i,u_j|G_1) + d(u_i,u_j|G_1)^2\right)
    \]
    and
    \[
    PWW(G_2) = \frac{1}{4}\sum_{k=1}^{k_2}\sum_{l=1}^{k_2} \left(d(v_k,v_l|G_1) + d(v_k,v_l|G_1)^2\right).
    \]
By Lemma \ref{DistOfProd} and  Corollary \ref{PeriOfProd}, we have
\begin{align}
 PWW(G_1 \times G_2) & = \frac{1}{4}\sum_{i=1}^{k_1}\sum_{j=1}^{k_1}\sum_{k=1}^{k_2}\sum_{l=1}^{k_2}  [d((u_i,v_k),(u_j,v_l)|G_1 \times G_2)  + d((u_i,v_k),(u_j,v_l)|G_1 \times G_2)^2 ]\nonumber\\ 
   & = \frac{1}{2}PW(G_1 \times G_2) + \frac{1}{4}\sum_{i=1}^{k_1}\sum_{j=1}^{k_1}\sum_{k=1}^{k_2}\sum_{l=1}^{k_2}d((u_i,v_k),(u_j,v_l)|G_1 \times G_2)^2. \label{PWWProdComp} 
  \end{align}
We now compute the second summand of (\ref{PWWProdComp}) that we will call $S$ temporarily
\begin{align*}
    S =  & \frac{1}{4}\sum_{i=1}^{k_1}\sum_{j=1}^{k_1}\sum_{k=1}^{k_2}\sum_{l=1}^{k_2}[d(u_i, u_j|G_1) + d(v_k, v_l)|G_2)]^2 \\
     = & \frac{1}{4}\sum_{i=1}^{k_1}\sum_{j=1}^{k_1}\sum_{k=1}^{k_2}\sum_{l=1}^{k_2}d(u_i, u_j|G_1)^2 + \frac{1}{4}\sum_{i=1}^{k_1}\sum_{j=1}^{k_1}\sum_{k=1}^{k_2}\sum_{l=1}^{k_2}d(v_k ,v_l|G_1)^2 \\
   &  +  \frac{1}{2}\sum_{i=1}^{k_1}\sum_{j=1}^{k_1}\sum_{k=1}^{k_2}\sum_{l=1}^{k_2} d(u_i, u_j|G_1)d(v_k ,v_l|G_1) \\
    = &\frac{k_2^2}{4}\sum_{i=1}^{k_1}\sum_{j=1}^{k_1}d(u_i, u_j|G_1)^2 + \frac{k_1^2}{4}\sum_{k=1}^{k_2}\sum_{l=1}^{k_2}d(v_k, v_l|G_2)^2  \\
   & + \frac{1}{2} \sum_{i=1}^{k_1}\sum_{j=1}^{k_1}\sum_{k=1}^{k_2}\sum_{l=1}^{k_2} d(u_i, u_j|G_1)d(v_k, v_l|G_2).
\end{align*}
From (\ref{PWWProdComp}) and Theorem \ref{PWProd}, we get the result.
\end{proof}
For $n\geq 2$, we denote by $G^n = \underbrace{G\times \cdots \times G}_{n \text{ times}}$ the $n$-fold cartesian product of $G$. Theorem \ref{PWWProd} allows us to get a formula for the peripheral hyper-Wiener index of the hypercube graph $Q_n = \underbrace{K_2 \times \cdots \times K_2}_{n \text{ times}}$ where $K_2$ is the complete graph of order $2$.
\begin{corollary}
    Let $n\geq 2$ and $Q_n$ be the hypercube graph. Then
    \[
    PWW(Q_n) = \sum_{i=1}^n 3^{n-i}2^{n+i-2}.
    \]
\end{corollary}
\begin{proof}
   We proceed by induction. The hypercube $Q_2 = K_2^2 = K_4$ is the complete graph of order $4$ so that $PWW(Q_2) = 10$. Assume  that our formula holds true for $Q_{n-1}$. From Theorem \ref{PWWProd} we have
   \[
   PWW(Q_n) = 6 PWW(Q_{n-1}) + 2^{2n-2}
   \]
   and by the induction hypothesis 
   \begin{align*}
   PWW(Q_n)& = 6\left(\sum_{i=1}^{n-1} 3^{n-i-1}2^{n+i-3}\right) + 2^{2n-2} \\
   & = \sum_{i=1}^n 3^{n-i}2^{n+i-2}.
   \end{align*}
   Therefore, the formula is true for $n\geq 2$.
\end{proof}

\section{Peripheral hyper-Wiener index of a tree}
In this section, we focus our attention on trees. It is well known that thanks to the acyclicity of trees, one can express the Wiener index (and other distance-based indices) in terms of quantities that are otherwise not well defined, as 
\[
W(T) = \sum_{e \in E(T)} n_1(e) \cdot n_2(e)
\]
where $n_1(e)$ and $n_2(e)$ are the number of vertices lying on the two sides of $e$ respectively. That idea was introduced by Wiener in his seminal paper \citep{Wie47}. One can consult \citep{WagWan18} section 2.2 for a proof and similar properties related to distances. In a similar way, one gets a formula for the hyper-Wiener index by replacing the summation over the edges by the summation over the paths. That is, we have
\begin{align*}
WW(T) & = \frac{1}{2} \sum_{\{u, v\}\subseteq V(T)}\left( d(u,v)+d(u,v)^2 \right) \\
& = \sum_{\{u, v\}\in V(T)} n_1(\pi_{uv}) \cdot n_2(\pi_{uv})
\end{align*}
where $\pi_{uv}$ is the the unique path joining $u$ to $v$, $n_1(\pi_{uv})$ and $n_2(\pi_{uv})$ are the number of vertices lying on two sides of $\pi_{uv}$ respectively. 

Following the same idea,  Narayankar and Lokesh \cite{NKL} proved the following result.
\begin{theorem}\label{PWTreeFormula}
    Let $T$ be a tree of order $n$ with $k$ peripheral vertices. Then
    \[
    PW(T) = \sum_{e\in E(T)}a_1(e) \cdot a_2(e)
    \]
    where $a_1(e)$ and $a_2(e)$ are the number of peripheral vertices of $T$ lying on the two sides of $e$.
\end{theorem}
Inspired by Theorem \ref{PWTreeFormula}, we calculate the peripheral hyper-Wiener index of an $n$-vertex tree with $k$ peripheral vertices.
\begin{theorem}
     Let $T$ be a tree of order $n$ with $k$ peripheral vertices. Then
     \[
     PWW(T) = \sum_{\{u, v\}\in V(T)} a_1(\pi_{uv}) \cdot a_2(\pi_{uv})
     \]
     where $a_1(\pi_{uv})$ and $a_2(\pi_{uv})$ are the number of peripheral vertices of $T$ lying on the two sides of $\pi_{uv}$.
\end{theorem}
\begin{proof}
    For $w, w'\in V(T)$, the quantity $\dfrac{1}{2}\left(d(w,w') + d(w,w')^2\right)$ counts the number of paths that lie on $\pi_{ww'}$. Therefore, we can calculate the peripheral hyper-Wiener index of $T$ by counting the number of times a path $\pi$ lies on a peripheral path and summing over all possible paths lying on $T$.
\end{proof}
We next provide bounds on  the peripheral hyper-Wiener index of trees in terms of the number of peripheral vertices and the diameter. 
\begin{theorem}
    Let $T$ be a tree of order $n$ and diameter $d$ with $k$ peripheral vertices. Then
    \[
    k{{d+2k-3}\choose{2}} \leq PWW(T) \leq 4{{d+1}\choose{2}}{{k}\choose{2}}.
    \]
    The bounds are sharp. The upper bound is attained if and only if every pair of peripheral vertices is at distance $d$  and the lower bound is attained if and only if $d=2$.
\end{theorem}
\begin{proof}
    Let $u_1, u_2, \cdots, u_k$ be the peripheral vertices of $T$. For all $i =1, 2, \cdots, k$, we have $\deg(u_i) = 1$ as the peripheral vertices are necessarily leaves. First, we prove the lower bound. By (\ref{PeriDistNr}) and the fact that for a peripheral vertex $u$, at least another peripheral vertex is at distance $d$ and the other $k-2$ remaining peripheral vertices are at least at distance $2$ from $u$, we have
    \begin{align*}
    PWW(T) & = \frac{1}{2}\sum_{i=1}^k \left( d_{P(G)}(u_i) + d_{P(G)}(u_i)^2 \right) \\
    & \geq \frac{1}{2} \sum_{i=1}^k \left( d + 2(k-2) + (d + 2(k-2))^2\right) \\
    & = k{{d+2k-3}\choose{2}}.
    \end{align*}
    Now, we look at the upper bound. Since the distance between two peripheral vertices is at most $d$, we have
    \begin{align*}
        PWW(T) & = \sum_{\{u, v\}\subseteq \Peri(G)} \left( d(u,v|T) + d(u,v|T)^2\right) \\
        & \leq \sum_{\{u, v\}\subseteq \Peri(G)} (d + d^2) \\
        & =   2 {{d+1}\choose{2}} \sum_{\{u, v\}\subseteq \Peri(G)} 1 \\
        & = 4{{d+1}\choose{2}} {{k}\choose{2}}.
    \end{align*}
   It is clear that the upper bound is attained if and only if every pair of peripheral vertices is at distance $d$. It remains to look at the optimality of the lower bound. One can easily see that if $k=2$ or $d=2$ then the lower bound is attained. Conversely, suppose $k\geq 3$ and $PWW(T) = \displaystyle k{{d+2k-3}\choose{2}}$. Then, for $i = 1, 2, \cdots, k$ at most one peripheral vertex is at distance $d$ from $u_i$ and the others are at distance $2$ from $u_i$. Since $k\geq 3$, necessarily $d=2$. 
\end{proof}

Now, we  determine the peripheral hyper-Wiener index of trees of diameter at most $4$. If $T$ has diameter $1$ then $PWW(T) = 1$. If $T$ has diameter $2$ then it is a star so that $PWW(T) = 3\displaystyle {{|V(T)|-1}\choose{2}}$. If $T$ has diameter $3$, then it is a double star $S_{m,n}$ with $m$ leaves on one side and $n$ leaves on the other side so that $PWW(T) = 6mn + 3m + 3n$.

We now look at diameter $4$ trees. Such a tree has a unique central vertex, say $u$, and we may assume that $T$ is  a rooted tree with root $u$. 
\begin{proposition}
    Let $T$ be a rooted tree of order $n$ with diameter $4$ and root $u$ which is its central vertex. Let $c_1, c_2, \cdots, c_s$ be the children of $u$  and for $i = 1, 2, \cdots, s$, let $C_i$ be the set of children of $c_i$. Then
    \[
    PWW(T) = 10 \sum_{1\leq i < j \leq s}|C_i||C_j| + 3\sum_{i=1}^s {{|C_i|}\choose{2}}.
    \]
\end{proposition}
\begin{proof}
We have that $\Peri(T) = \displaystyle \bigcup_{i=1}^s C_i $. Therefore,  
\begin{align*}
    PWW(T) & = \frac{1}{2}\sum_{\{u,v\}\subseteq \displaystyle \cup_{i=1}^s C_i }\left(d(u,v|T) + d(u,v|T)^2\right) \\
    & = \frac{1}{2}\sum_{i=1}^s \sum_{\{u,v\}\subseteq C_i}\left(d(u,v|T) + d(u,v|T)^2\right) + \frac{1}{2}\sum_{1\leq i < j \leq s}\sum_{u\in C_i, v\in C_j}\left(d(u,v|T) + d(u,v|T)^2\right) \\
    & =\frac{1}{2} \sum_{i=1}^s \sum_{\{u,v\}\subseteq C_i} 6 + \frac{1}{2} \sum_{1\leq i < j \leq s}\sum_{u\in C_i, v\in C_j} 20 \\
    & = 3 \sum_{i=1}^s{{|C_i|}\choose{2}} + 10 \sum_{1\leq i < j \leq s}|C_i||C_j|.
\end{align*}
\end{proof}
The following lemma on the diameter of the complement of a graph with large diameter is well known, we give a quick proof of it.
\begin{lemma}\label{DiamComp}
    Let $G$ be a graph such that $\diam(G) \geq 4$. Then $\diam(\overline{G})\leq 2$ where $\overline{G}$ is the complement of $G$. 
\end{lemma}
\begin{proof}
    If $G$ is not connected then the result is obvious, so we may assume $G$ is connected. Let $u, v \in V(G)$ be such that $d(u,v|G) \geq 4$. Let $x, y\in V(G)$. If $|\{u,v,x,y\}|< 4$, the result is clear. So we assume that $|\{u,v,x,y\}|=4$. If $xy \notin E(G)$ then $d(x,y|\overline{G}) = 1$ and we are done. So we assume now that $xy \in E(G)$. If there is no edge of $G$ from $\{x, y\}$ to $\{u, v\}$ then $d(x,y|\overline{G})\leq d(x,u|\overline{G}) + d(u,y|\overline{G}) = 2$ and we are done. If there is at least an edge of $G$ that connects $\{x, y\}$ to $\{u, v\}$, say $ux\in E(G)$, then $vx, vy \notin E(G)$ so that $d(x,u|\overline{G})\leq d(x,v|\overline{G}) + d(v,y|\overline{G}) =2$ and we are done.
\end{proof}
\begin{theorem}
    Let $T$ be a tree of order $n$ and let $\overline{T}$ be its connected complement. Then
    \begin{enumerate}
        \item $PWW(\overline{T}) = 6$ if and only if $\diam(T) = 3$.
        \item $PWW(\overline{T}) = \dfrac{n^2 + 3n - 4}{2}$ if and only if $\diam(T) > 3$.
    \end{enumerate}
\end{theorem}
\begin{proof}
\begin{enumerate}
        \item \label{DiamThreeChar1} If $\diam(T) = 3$ then $T$ is a double star $S_{m,n}$ with two central vertices, say $u$ and $v$, with $m$ and $n$ leaves on the two sides of the two central vertices respectively. Therefore, $\Peri(\overline{T}) = \{u, v\}$ so that $PWW(\overline{T}) = d(u,v|\overline{T}) = 6$.  Conversely, suppose that $PWW(\overline{T}) = 6$. For the sake of contradiction, suppose also that $\diam(T) \neq 3$. It is clear that $\diam(T) \geq 4$, otherwise $\overline{T}$  would be the complement of either $K_2$ or a star graph and would be disconnected. By Lemma \ref{DiamComp}, it follows that $\diam(\overline{T}) \leq 2$. As $T$ is connected, we must have $e_{\overline{T}}(v) = 2$ for any $v\in V(\overline{T})$ since any vertex with eccentricity $1$ in $\overline{T}$ would be isolated in $T$. Therefore, $\overline{T}$ is self-centered and peripheral of diameter $2$. By Theorem \ref{DiamTwoThm} we have
        \[
        PWW(\overline{T}) = 2\displaystyle {{n}\choose{2}} + \displaystyle {{k}\choose{2}} -2m
        \]
        where $k=n$ is the number of peripheral vertices and $m = \displaystyle{{n}\choose{2}} - (n-1)= {{n-1}\choose{2}}$ the size of $\overline{T}$. Hence,
        \[
        PWW(\overline{T}) = 3\displaystyle {{n}\choose{2}} -2 \displaystyle {{n-1}\choose{2}}
        \]
        which implies that $6 = \dfrac{n^2 + 3n - 4}{2}$, which is impossible. Therefore $\diam(T) = 3$.
        \item If $\diam(T) > 3$, we have from the above argument that $PWW(\overline{T}) = \dfrac{n^2 + 3n - 4}{2}$. Conversely, let $PWW(\overline{T}) = \dfrac{n^2 + 3n - 4}{2}$. Since $\overline{T}$ is the connected complement of $T$, $\diam(\overline{T}) \geq 3$. By \ref{DiamThreeChar1}, we are led to a contradiction. Hence $\diam(T) >3$.
    \end{enumerate}
\end{proof}

We end this section by looking at the peripheral hyper-Wiener index of caterpillars and some deformation of caterpillars. A \textit{caterpillar} is a tree such that removing the leaves and the incident edges to the leaves produce a path graph, called its \textit{spine}. The \textit{code} of a caterpillar $C$  with spine $S_C = u_1 u_2 \cdots u_s$ is the ordered $s$-tuple $(c_1, c_2, \cdots, c_s)$ where $c_i$ is the number of leaves adjacent to $u_i$ for $i\in \{1, 2, \cdots, s\}$. Note that by definition $c_1 \neq 0$ and $c_s \neq 0$.
\begin{theorem}\label{PWWCat}
    Let $C$ be a caterpillar with spine $S_C = \{u_1, u_2, \cdots, u_s\}$ and code $(c_1, c_2, \cdots, c_s)$. Then, 
    \[
    PWW(C) = 3{{c_1}\choose{2}} + 3{{c_s}\choose{2}}  + \frac{c_1c_s}{2}(s+1)(s+2).
    \]
\end{theorem}
\begin{proof} 
    Let $\Peri(C) = \{\omega_1, \cdots, \omega_{c_1}, \gamma_1, \cdots, \gamma_{c_s} \}$ be the set of peripheral vertices of $C$ where the leaves $\omega_i$ are the ones attached to $u_1$ and the leaves $\gamma_j$ are the ones attached to $u_s$. Therefore,
    \begin{align*}
    PWW(C)  = &\frac{1}{2}\sum_{1\leq i <j \leq c_1} \left(d(\omega_i, \omega_j)+d(\omega_i, \omega_j)^2 \right) + \frac{1}{2}\sum_{1\leq r <k \leq c_s} \left(d(\gamma_i, \gamma_j)+d(\gamma_i, \gamma_j)^2 \right) \\
    & + \frac{1}{2} \sum_{1\leq i \leq c_1}\sum_{1\leq j\leq c_s} \left(d(\omega_i, \gamma_j) + d(\omega_i, \gamma_j)^2 \right) \\
    = &  \frac{1}{2}\sum_{1\leq i <j \leq c_1} 6  + \frac{1}{2}\sum_{1\leq r <k \leq c_s} 6  + \frac{1}{2} \sum_{1\leq i \leq c_1}\sum_{1\leq j\leq c_s} \left((s+1) + (s+1)^2 \right) \\
    = & 3 {{c_1}\choose{2}} + 3{{c_s}\choose{2}} + \frac{c_1c_s}{2}(s+1)(s+2).
    \end{align*}
\end{proof}
Now, let $T$ be a tree obtained from a caterpillar $C$ with spine $S_C = u_1 u_2 \cdots u_s$ and code $(c_1, 0, c_3, \cdots, c_s)$ by adding a star $K_{1,c}$ and joining its center to $u_2$. Such a graph is a particular case of tree called \textit{lobster}.
\begin{theorem}
    Let $T$ be a tree obtained as above. Then
    \begin{align*}
         PWW(T) =   3{{c_1}\choose{2}} + 3{{c_s}\choose{2}}  + 3{{c}\choose{2}} + 10 c_1c + c_s(c_1 + c) (s+1)(s+2).
    \end{align*}
    
\end{theorem}
\begin{proof}
    Similar to the above theorem by noticing that $\Peri(T) = \{\omega_1, \cdots, \omega_{c_1}, \gamma_1, \cdots, \gamma_{c_s} ,z_1, \cdots, z_c\}$ where the $z_i$'s are the leaves from the star $K_1,c$.
\end{proof}
\section*{Acknowledgements}
We thank Dr. T. P. Chalebgwa of the University of Pretoria who suggested the idea of the peripheral hyper-Wiener index of a graph.
\bibliographystyle{alpha}
\bibliography{backmatter/biblio.bib}

\end{document}